\documentclass[12pt]{amsart}

\usepackage[utf8]{inputenc}
\usepackage[english]{babel}
\usepackage{amsmath}
\usepackage{amsfonts}
\usepackage{amssymb}
\usepackage{graphicx}
\usepackage{mathtools}
\usepackage{amsthm}

\newtheorem{thm}{Theorem}[section]
\newtheorem{prop}[thm]{Proposition}
\newtheorem{lem}[thm]{Lemma}
\newtheorem{cor}[thm]{Corollary}
\newtheorem{rem}[thm]{Remark}
\newtheorem{ass}{Assumption}

\theoremstyle{definition}
\newtheorem{defi}{Definition}[section]
\newcommand{\R}{\mathbb{R}}
\renewcommand{\P}{\mathbb{P}}
\newcommand{\N}{\mathbb{N}}

\newcommand{\C}{\mathbb{C}}

\newcommand{\E}{\mathbb{E}}

\newcommand{\var}{\mathrm{var}}
\newcommand{\supp}{\mathrm{supp}}
\newcommand{\cK}{\mathcal{K}}
\newcommand{\cP}{\mathcal{P}}
\newcommand{\La}{\Lambda}
\newcommand{\locally}{\mathcal{I}_{1,\mathrm{loc}}}
\newcommand{\tr}{\mathrm{Tr}}

\title{Number variance for homogeneous determinantal processes on hyperbolic spaces}
\author{Pierre Lazag}

\usepackage{fullpage}

\usepackage{hyperref}
\begin{document}

\date{}
\maketitle

\address{LAREMA, UMR CNRS 6093, 2 Boulevard Lavoisier 49045 Angers cedex 01, France}

\email pierrelazag@hotmail.fr 

\normalsize

\begin{abstract}
    We consider an abstract determinantal point process on a general non--elementary Gromov hyperbolic metric space governed by an orthogonal projection in the case when the space is homogeneous and the point process is invariant under isometries. We give a lower bound of the variance of the number of points inside a ball that is proportional to the volume of the ball. In particular, such point processes are never hyperuniform. Our result applies to the known examples of radial determinantal point processes on Cayley trees and on the standard hyperbolic spaces governed by Bergman projection kernels. 
\end{abstract}

\section{Introduction}
Determinantal point processes are random point fields for which the correlation functions are expressed by the determinant of a certain kernel. Introduced by O. Macchi in \cite{Macchi75}, they have been the subject of a systematic study in \cite{Sos00}, \cite{ST03a}, \cite{HKPV09}, and appear in many different contexts such as random matrix theory, combinatorics, representation theory, mathematical physics, see e.g. \cite{borodindet} and references therein as well as the references cited above.\\

By the Macchi-Soshnikov/Shirai-Takahashi Theorem (\cite{Sos00}, \cite{ST03a}, see also \cite{HKPV09}), see Theorem \ref{thm:machhisoshnikov} below, the kernel of any locally trace class orthogonal projection serves as a kernel of a determinantal point process, and one could ask about the general properties of such a point processes. For instance, the total number of points of such a point process is deterministic, and is equal to the rank of the projection. We shall also see that the variance of the number of points inside a bounded set is always dominated by the expectation of the number of points. We consider here infinite rank orthogonal projections acting on the space of square integrable functions on a hyperbolic metric space in the sense of Gromov, see \cite{gromov}, \cite{ghys-delaharpe}, and ask the question of the behavior of the variance of the number of point inside a ball.

We make the assumption that the underlying space is homogeneous in the sense that the group of its isometries acts transitively and leaves the reference measure invariant, and that the point process is invariant under isometries. Such a framework encompasses the  \emph{continuous} cases of determinantal point processes on the Poincar\'e disk governed by Bergman projections appearing first in \cite{peres-virag} as the zero set of Gaussian analytic functions and further generalized in \cite{Kri09} for weighted Bergman spaces as well as in \cite{demni-lazag} for poly-analytic Bergman spaces and in \cite{bufetov-qiu-bergman} in higher dimensions, as well as the \emph{discrete} cases of radial determinantal point processes on Cayley trees defined and characterized in \cite{qiu-trees}.

Whereas for all translation invariant determinantal point processes on the Euclidean space $\R^d$ the variance of the number of points inside a large ball is asymptotically negligible compared to the expectation of the number of points, see \cite{KLS}, \cite{Torquato2018}, \cite{Sos00}, the main result of this paper, Theorem \ref{thm:thm1} below, gives a non-asymptotic lower bound for this variance that is proportional to the expectation. In our understanding, it is remarkable that, although the assumption of homogeneity of the process might seem to be strong, the assumption of hyperbolicity we use, that is quite coarse and general and applies both to the continuous and discrete cases, is sufficient to force determinantal point processes with symmetries to be chaotic, in the sense that they have a large variance, regardless of the analytic properties of their correlation kernel. In particular, hyperbolicity prevents from hyperuniformity, i.e. from having a growth of the variance that is negligible compared to the growth of the expectation, see \cite{Torquato2018} for a survey of the notion of hyperuniformity, as well as for more details about the expected regimes for the growth of the variance for translation invariant point processes in Euclidean spaces.\\

We describe below our main result with more precision.
\subsection{Main result}
Let $(S,d)$ be a proper complete and separable metric space. We equip $S$ with the Borel sigma-algebra induced by the distance $d$, and we let $\lambda$ be a Radon measure on $S$. For $x \in S$ and $R>0$, we denote by $B_R(x)$ the open ball centered at $x$ and of radius $R$. We fix a base point $o \in S$ and simply denote by $B_R$ the ball $B(o,R)$. We will make the following geometric assumptions on $(S,d,\lambda)$, and we refer to section \ref{sec:geometry} below for precise definitions.
\begin{ass} \label{ass:lengthspace} The space $(S,d)$ is a length space.
\end{ass}
\begin{ass}\label{ass:deltahyperbolic} There exists $\delta \geq 0$ such that the space $(S,d)$ is $\delta$-hyperbolic.
\end{ass}
\begin{ass}\label{ass:expgrowth}The balls in the space $(S,d,\lambda)$ have an exponentially increasing size.
\end{ass}
Our main result is then the following Theorem. We refer to section \ref{sec:dpp} below for notation and definition of determinantal point processes, and to section \ref{sec:homogenous} for the notion of homogeneous spaces and homogeneous determinantal point processes.
\begin{thm} \label{thm:thm1}Let $\P$ be a homogeneous determinantal point process on the homogeneous space $(S,d,\lambda)$. If the space $(S,d)$ satisfies assumptions \ref{ass:lengthspace}, \ref{ass:deltahyperbolic} and \ref{ass:expgrowth}, then, there exists $C>0$ such that for all $R>0$, we have
\begin{align}
\frac{\var( \Xi(B_R))}{\E [ \Xi(B_R) ]} \geq C.
\end{align}
\end{thm}
The constant $C$ will be given in section \ref{sec:proof} below, see the formula (\ref{eq:constant}). Precise asymptotic formulas for the variance in hyperbolic contexts have been given in \cite{peres-virag}, \cite{demni-lazag}, \cite{bufetov-qiu-bergman}. We here stay at a qualitative and coarse level and lose in precision what we gain in generality.\\

It is worth noting that the Poisson point process, which fully uncorrelated and in a sense a degenerate determinantal point process, has the variance that is equal to the expectation. Our result can be interpreted as the fact that determinantal point processes on hyperbolic space, although having strong correlations, resemble the Poisson point process, and is thus in accordance with the result from \cite{bufetov-fan-qiu} where it is shown that the analytic properties of the Bergman space the authors consider imply that the corresponding determinantal point processes are insertion-deletion tolerant.
\subsection{Scheme of the proof and organisation of the paper}
We first recall the definition and general properties of determinantal point processes in section \ref{sec:dpp}.

Our proof first lies then on the interpretation of the variance in terms of the volume of a lunule in the formula given in \ref{prop:hyperboliclunule} below and that holds in the general context of homogeneous determinantal point processes. Such a formula has been used in\cite{demni-lazag} and \cite{KLS}, and is stated and proved here in greater generality in section \ref{sec:homogenous}.

We then show how to control the volume of the lunule in the hyperbolic context in section \ref{sec:geometry}, and finally conclude the proof in section \ref{sec:proof}.

\subsection{Acknowledgements}I would like to thank Adrien Boulanger for helpful discussions concerning hyperbolic geometry. This project recieved fundings from project ULIS 2023-09915 from R\'egion Pays de la Loire.
\section{Determinantal point process} \label{sec:dpp}
We refer to \cite{daley-verejones} for generalities on point processes. A configuration $\Xi$ on $S$ is a non-negative integer valued Radon measure on $S$. The space of configurations, denoted by $\text{Conf}(S)$, is again a complete separable metric space, and we equip it with its Borel sigma-algebra. A configuration $\Xi \in \text{Conf}(S)$ is said to be {\it simple} if one has $\Xi(\{x\}) \in \{0,1\}$ for any $x \in S$.
\begin{defi}A {\it point process} on $S$ is a probability measure on $\text{Conf}(S)$. A point process is simple if it is supported on simple configurations.
\end{defi}
For a simple point process $\P$ on $S$ with the reference measure $\lambda$ and a positive integer $n \in \N$, the $n$-th correlation function $\rho_n : S^n \rightarrow \C$ of $\P$, if it exists, is defined by :
\begin{multline}\label{eq:correl}\int_{\text{Conf}(S)} \overset{*}{\sum_{x_1,...,x_n \in \supp (\Xi)}}f(x_1,...,x_n) d\P(\Xi) \\
 = \int_{S^n} f(x_1,...,x_n) \rho_n(x_1,...,x_n) d\lambda(x_1)...d\lambda(x_n),
\end{multline}
for any measurable compactly supported function $f : S^n \rightarrow \C$. Here, the sum $\overset{*}{\sum}$ is taken over all pairwise distinct points $x_1,...,x_n \in \supp(\Xi)$. The sequence of correlation functions of a point process $\P$ characterizes $\P$.
\begin{defi}\label{def:dpp}A point process $\P$ on $S$ is a {\it determinantal point process } if there exists a kernel :
\[ K : S \times S \rightarrow \C.
\]
such that, for any $n \in \N$, the $n$-th correlation function exists and is given by :
\[\rho_n(x_1,...,x_n) = \det \left( K(x_1,x_n) \right)_{i,j=1}^n.
\]
The kernel $K$ is called the \emph{correlation kernel} of the point process $\P$.
\end{defi}
Such a kernel defines an integral operator $\cK$ :
\begin{align*} \cK : \hspace{0.1cm} L^2(S,\lambda) &\rightarrow L^2(S,\lambda) \\
f &\mapsto \left( x \mapsto \int_S f(y)K(x,y) d\lambda(y) \right).
\end{align*}
For a bounded set $\Lambda \subset S$, we denote by $\cP_\Lambda$ the restriction operator onto $\Lambda$, i.e. the operator of multiplication by the indicator function of $\La$, denoted by $\mathfrak{1}_\La$. An integral operator $\cK$ with kernel $K$ is locally of trace class if for any bounded Borel set $\Lambda \subset S$, the operator $\cP_\Lambda \cK \cP_\Lambda$ is of trace class, i.e.:
\[  \int_\Lambda K(x,x) d\lambda(x) <+ \infty. \]
We denote by $\locally(S,\lambda)$ the ideal of locally trace class operators on $L^2(S,\lambda)$. All locally trace class operators admit a kernel. We recall the following fundamental result (see \cite{Sos00}, \cite{ST03a}, \cite{HKPV09}):
\begin{thm} \label{thm:macchisoshnikov} Let $\cK \in \locally(S,\lambda)$ be a hermitian locally trace class operator. Then, its kernel $K$ serves as the kernel of a determinantal point process if and only if $0 \leq \cK \leq I$. In particular, any locally trace class orthogonal projector onto a closed subspace $ L \subset L^2(S,\lambda)$ gives rise to a determinantal point process.
\end{thm}
For a determinantal point process $\P$ with kernel $K$ and associated integral operator $\cK$, we have by definition \ref{def:dpp} and formula (\ref{eq:correl}) :
\begin{equation} \label{eq:expectation}
\E [ \Xi (\Lambda) ] = \int_\Lambda K(x,x)d\lambda(x) =\tr ( \cP_\Lambda \cK \cP_\Lambda )
\end{equation}
where here and in the sequel, $\E$ denotes the expectation with respect to $\P$. When $\cK$ is an orthogonal projector, the reproducing property:
\[ \int_S K(x,y) K(y,z)d\lambda(y) = K(x,z)\]
entails the following formula for the number variance:
\begin{multline}\label{eq:variance}
\var( \Xi(\Lambda) ):= \E[ \left(\Xi(\Lambda)- \E[\Xi(\Lambda)] \right)^2] 
= \frac{1}{2} \int_S \int_S | \mathfrak{1}_\La(x) - \mathfrak{1}_\La(y)| |K(x,y)|^2 d\lambda(x)d\lambda(y).
\end{multline}
The latter equality (\ref{eq:variance}) can be rewritten as :
\begin{equation}\label{eq:variance2}
\var(\Xi(\Lambda))= \int_\Lambda \int_{S \setminus \Lambda} |K(x,y)|^2 d\lambda(x) d\lambda(y).
\end{equation}
\begin{rem}
    Observe that the triangular inequality together with the reproducing property imply that the variance (\ref{eq:variance}) is always dominated by the expectation.
\end{rem}
\section{Homogeneous determinantal processes} \label{sec:homogenous}
The goal of this section is to introduce the notion of homogeneous determinantal point processes on homogeneous spaces. Its main result is proposition \ref{prop:variance} below, which expresses the number variance in terms of the volume of lunules, i.e. intersection of balls. We start by defining the notion of homogeneous spaces.
\begin{defi}The metric measured space $(S,d,\lambda)$ with base point $o$ is called homogeneous if:
\begin{enumerate}
\item[$\bullet$] the group of isometries $G = \text{Isom}(S)$ acts transitively on $S$;
\item[$\bullet$] the measure $\lambda$ is invariant by $G$.
\end{enumerate}
\end{defi}
A homogeneous determinantal point process is then defined as follows. We will always assume that its correlation kernel is the kernel of an orthogonal projection.
\begin{defi}Let $\P$ be a determinantal point process with kernel $K$ on a homogeneous space $(S,d,\lambda)$. We say that $\P$ is homogeneous if the modulus of its kernel is $G$-invariant :
\begin{align} \label{cond:homogenous}
|K(g.x,g.y)|=|K(x,y)|, \quad \forall g \in G, \hspace{0.1cm}\forall x,y \in S.
\end{align}
\end{defi}
\begin{rem}It is easily seen that the distribution of a homogeneous point process $\P$ is invariant by the action of the group $G$, as all the minors of the hermitian kernel $K$ are invariant by condition (\ref{cond:homogenous}).
\end{rem}
\begin{rem}If the kernel $K$ is radial, i.e. if its values $K(x,y)$ only depends on $d(x,y)$, then the corresponding point process is homogeneous. We do not, however, need this more restrictive assumption.
\end{rem}
The first property we can exhibit for a homogeneous determinantal point process, is that the mean number inside a subset of $S$ is proportional to the volume of this subset. We fix any base point $o \in S$, for which the choice is not relevant as the space is homogeneous.
\begin{prop}\label{prop:expectation} Assume that $(S,d,\lambda)$ is homogeneous, and let $\La \subset S$ be a measurable bounded subset of $S$. If $\P$ is a homogeneous determinantal point process on $S$ with kernel $K$, then we have:
\begin{align*}
\E[ \Xi (\La) ] = K(o,o) \lambda(\La).
\end{align*}
\end{prop}
\begin{proof}Indeed, if $ x \in S$ and $g \in G$ is such that $g.x=o$, then
\[0 \leq K(o,o)= K(g.x,g.x)=K(x,x) \] does not depend on $x$, so that
\begin{align*}
    \E [ \Xi (\Lambda)] = \int_\Lambda K(x,x)= K(o,o) \lambda(\Lambda).
\end{align*}
\end{proof}
More interesting is the expression of the number variance in geometric terms, given by the following key proposition. 
\begin{prop}\label{prop:variance}Let $(S,d,\lambda)$ be a homogeneous space. If $\P$ is a homogeneous determinantal point process on $S$ with kernel $K$, then, for any $R>0$, one has :
\begin{align*}
\var(\Xi(B_R)) = \int_S |K(x,o)|^2 \lambda(B_R^c\cap B_R(x)) d\lambda(x),
\end{align*}
where $B_R^c = S \setminus B_R$ is the complementary set of $B_R$.
\end{prop}
\begin{proof}
We first have from formula (\ref{eq:variance2}) that :
\begin{align*}
\var(\Xi(B_R))=\int_{y \in B_R^c} \int_{x \in B_R} |K(x,y)|^2d\lambda(x)d\lambda(y).
\end{align*}
Since $G$ acts transitively on $S$, one can choose, for $y\in B_R^c$, an isometry $g \in G$ such that $g.y=o$. Performing the change of variable $x \mapsto g.x$ and since $\lambda$ is $G$-invariant, we obtain :
\begin{align*}
\var(\Xi(B_R))=\int_{y \in B_R^c} \int_{x \in B_R(y)} |K(g^{-1}.x,g^{-1}.o)|^2 d\lambda(x) d\lambda(y).
\end{align*}
Since $\P$ is homogeneous, the latter equality reads:
\begin{align*}
\var(\Xi(B_R))=\int_{y \in B_R^c} \int_{x \in B_R(y)} |K(x,o)|^2 d\lambda(x) d\lambda(y).
\end{align*}
We now apply Fubini's Theorem. The variable $x$ will range over $\cup_{y \in B_R^c} B_R(y)=S\setminus \{o\}$, while for the variable $y$, we have $y \in B_R^c$ and $ x \in B_R(y)$ if and only if $y \in B_R^c$ and $y \in B_R(x)$. We thus obtain :
\begin{align*}
\var(\Xi(B_R))=\int_S |K(x,o)|^2\lambda(B_R^c \cap B_R(x)) d\lambda(x),
\end{align*}
which is the desired result.
\end{proof}
\section{The geometry of $S$}\label{sec:geometry}
We here precisely define the geometric setting where our result holds, given by assumptions \ref{ass:lengthspace}, \ref{ass:deltahyperbolic} and \ref{ass:expgrowth}. The results of this section are proposition \ref{prop:hyperboliclunule} which holds on general $\delta$-hyperbolic spaces, together with its consequence when the balls have an exponential size, given in corollary \ref{cor:exp}. We refer to \cite{ghys-delaharpe} for general notions of hyperbolicity in the sense of Gromov first defined in \cite{gromov}.
\begin{defi}Let $x,y \in S$. A \emph{geodesic arc from $x$ to $y$} is an isometry $\gamma_{x,y} : [0,d(x,y)] \subset \R \rightarrow S$ such that $\gamma_{x,y}(0)=x$ and $\gamma_{x,y}(d(x,y))=y$. We say that $(S,d)$ is a \emph{length space} if for all $x,y \in S$, there exists a geodesic arc from $x$ to $y$.
\end{defi}
Assume from now that $(S,d)$ is a length space. For $x,y \in S$, we may identify the geodesic $\gamma_{x,y}$ with its image and denote it by $[x,y]$. We now introduce the notion of $\delta$-hyperbolic spaces.
\begin{defi}Let $\delta \geq 0$ be a non-negative real number. We say that $(S,d)$ is \emph{$\delta$-hyperbolic} if, for any $x,y,z \in S$ and any $p \in [x,y]$, there exists $q \in [y,z] \cup [x,z]$ such that $d(p,q) \leq \delta$.
\end{defi}
We finally precisely define the notion of exponential size of balls.
\begin{defi}We say the space $(S,d,\lambda)$ has exponential size of balls if there exists $c, \alpha >0$ such that, for any $x \in S$ and any $R>0$, we have :
\begin{align*}
c^{-1}e^{\alpha R} \leq \lambda(B_R(x)) \leq c e^{\alpha R}.
\end{align*}
\end{defi}
We first state the main proposition of this section.
\begin{prop} \label{prop:hyperboliclunule}Let $\delta \geq 0$ and assume that $(S,d)$ is a $\delta$-hyperbolic length space. Let $R \geq 0$ and let $x, y \in S$. Set $r:= d(x,y)$ and let :
\begin{align*} p:= \gamma_{x,y} \left( \frac{r}{2} \right) \in [x,y]
\end{align*}
be the "middle point" between $x$ and $y$. Then we have :
\begin{align*}
B(x,R) \cap B(y,R) \subset B(p,R-r/2 +2\delta).
\end{align*}
\end{prop}
\begin{proof}
Let $z \in B(x,R) \cap B(y,R)$. Considering the geodesic triangle formed by the points $x,y$ and $z$, and using the $\delta$-hyperbolicity, we can choose a point $q \in [x,z] \cup [y,z]$ such that
 \begin{align} \label{ineq:hyperbolic0}
 d(p,q) \leq \delta.
 \end{align} Assume for example that $q \in [x,z]$. By the triangular inequality and using inequality (\ref{ineq:hyperbolic0}), we have :
\begin{align} \label{ineq:hyperbolic1}
d(p,z) \leq \delta + d(q,z).
\end{align}
On the other hand, we also have from the triangular inequality  and from (\ref{ineq:hyperbolic0}) that :
\begin{align} \label{ineq:hyperbolic2}
d(x,q) \geq d(x,p) -d(p,q) \geq r/2 -\delta.
\end{align}
Since $q \in [x,z]$, and since $d(x,z) \leq R$ we have :
\begin{align} \label{eq:hyperbolic0}
d(q,z)=d(x,z) -d(x,q) \leq R - d(x,q).
\end{align}
Plugin inequality (\ref{ineq:hyperbolic2}) into (\ref{eq:hyperbolic0}), one obtains :
\begin{align} \label{ineq:hyperbolic3}
d(q,z) \leq R + \delta -r/2.
\end{align}
Plugin the latter inequality (\ref{ineq:hyperbolic3}) into (\ref{ineq:hyperbolic0}), we finally have :
\begin{align*}
d(p,z) \leq R -r/2 + 2\delta,
\end{align*}
and the proof is complete.
\end{proof}
\begin{cor}\label{cor:exp} If $(S,d)$ is a $\delta$-hyperbolic length space and if besides $(S,d,\lambda)$ has exponential size of balls with constants $c,\alpha>0$, then, we have :
\begin{align} \label{ineq:vollunulehyperbolic}
\lambda(B_R^c \cap B_R(x)) \geq c^{-2} \lambda(B_R) \left( 1- e^{-\alpha (d(o,x)/2 - 2 \delta -2 \log(c) / \alpha)}\right) \mathfrak{1}_{B_{4 \left( \delta + \frac{\log (c)}{\alpha}\right)}^c}(x).
\end{align}
\end{cor}
\begin{proof}
By proposition \ref{prop:hyperboliclunule} and from the upper bound given by the exponential size of balls, we have :
\begin{align*}
\lambda(B_R \cap B_R(x)) \leq ce^{\alpha (R- d(o,x)/2 + 2 \delta)},
\end{align*}
and thus, using the lower bound, we have :
\begin{align*}
\lambda(B_R^c \cap B_R(x) ) &= \lambda(B_R(x))- \lambda( B_R \cap B_R(x)) \\
 & \geq c^{-1}e^{\alpha R} - c e^{\alpha(R - d(o,x)/2 + 2 \delta)} \\
 &= c^{-1}e^{\alpha R} \left( 1 - e^{ -\alpha (d(o,x)/2 - 2 \delta -2 \log(c) / \alpha)} \right)\\
 & \geq c^{-2} \lambda(B_R)\left( 1 - e^{ -\alpha (d(o,x)/2 - 2 \delta -2 \log(c) / \alpha)} \right).
\end{align*}
Since $\lambda(B_R^c \cap B_R(x) ) \geq 0$, one can write :
\begin{align*}
\lambda(B_R^c \cap B_R(x) ) &\geq \max \left\{ c^{-2} \lambda(B_R)\left( 1 - e^{ -\alpha (d(o,x)/2 - 2 \delta -2 \log(c) / \alpha)} \right), 0  \right\} \\
&= c^{-2} \lambda(B_R) \left( 1- e^{-\alpha (d(o,x)/2 - 2 \delta -2 \log(c) / \alpha)}\right) \mathfrak{1}_{B_{4 \left( \delta + \frac{\log (c)}{\alpha}\right)}^c}(x),
\end{align*}
where the last equality follows from the fact that :
\begin{align} \label{ineq:pos}
 e^{-\alpha (d(o,x)/2 - 2 \delta -2 \log(c) / \alpha)} \leq 1 \Leftrightarrow d(o,x) \geq 4 \left( \delta + \frac{\log (c)}{\alpha} \right). 
\end{align}
The proof is complete.
\end{proof}
\section{Proof of Theorem \ref{thm:thm1}}\label{sec:proof}
The Theorem directly follows from Corollary \ref{cor:exp}, proposition \ref{prop:variance} and proposition \ref{prop:expectation}. Indeed, multiplying inequality (\ref{ineq:vollunulehyperbolic}) from Corollary \ref{cor:exp} by $|K(o,x)|^2$ and integrating over $x \in S$, one obtains by proposition \ref{prop:variance} that :
\begin{multline*}
\var ( \Xi(B_R) ) \geq \\
 c^{-2} \lambda(B_R) \int_{S \setminus B_{4 \left( \delta + \log(c) /\alpha \right)}} |K(o,x)|^2 \left( 1- e^{-\alpha (d(o,x)/2 - 2 \delta -2 \log(c) / \alpha)}\right) d\lambda(x).
\end{multline*}
Dividing by $\E [ \Xi(B_R) ]$ and applying proposition \ref{prop:expectation}, we are left with :
\begin{align*}
\frac{\var (\Xi( B_R))}{\E [ \Xi (B_R)]} \geq C,
\end{align*}
with :
\begin{align} \label{eq:constant}
C:=\frac{1}{c^2 K(o,o)} \int_{S \setminus B_{4 \left( \delta + \log(c) /\alpha \right)}} |K(o,x)|^2 \left( 1- e^{-\alpha (d(o,x)/2 - 2 \delta -2 \log(c) / \alpha)}\right) d\lambda(x).
\end{align}
The constant $C$ is positive by inequality (\ref{ineq:pos}) and since $\lambda \left( S \setminus B_{4 \left( \delta + \log(c) /\alpha \right)} \right)= + \infty$ (which for example follows from the exponential size of balls). Theorem \ref{thm:thm1} is proved completely.

\end{document}